\documentclass[12pt]{article}%

\usepackage{amsmath,enumerate}
\usepackage{amsfonts}
\usepackage{amssymb}
\usepackage{multicol}
\usepackage{epic,eepic,color}
\newtheorem{theorem}{Theorem}[section]

\newtheorem{lemma}[theorem]{Lemma}

\newenvironment{proof}[1][Proof]{\noindent\textbf{#1.} }
{\hfill \ \rule{0.5em}{0.5em}}

\newcommand{\nn}{\mathbb{N}}

\begin{document}

\title{A path Tur\'{a}n problem for infinite graphs}
\author{Xing Peng \thanks{Center for Applied Mathematics,  Tianjin University, Tianjin  300072,   China, \texttt{x2peng@tju.edu.cn}} \and Craig Timmons\thanks{Department of Mathematics and Statistics, California State University Sacramento, \texttt{craig.timmons@csus.edu}. This work was supported by a grant from the Simons Foundation (\#359419, Craig Timmons).}}
\date{}

\maketitle

\begin{abstract}
Let $G$ be an infinite graph whose vertex set is the set of positive integers, and let $G_n$ be the subgraph of
$G$ induced by the vertices $\{1,2, \dots , n \}$.  An increasing path of length $k$ in $G$, denoted $I_k$, is a sequence
of $k+1$ vertices $1 \leq i_1 < i_2 < \dots < i_{k+1}$ such that $i_1, i_2, \ldots, i_{k+1}$ is a path in $G$.
For $k \geq 2$, let $p(k)$ be the supremum of
$\liminf_{ n \rightarrow \infty} \frac{ e(G_n) }{n^2}$ over all $I_k$-free graphs $G$.  In 1962, Czipszer, Erd\H{o}s, and Hajnal proved that
$p(k) = \frac{1}{4} (1  - \frac{1}{k})$ for $k \in \{2,3 \}$.  Erd\H{o}s conjectured that this holds for all $ k \geq 4$.  This was disproved
for certain values of $k$ by Dudek and R\"{o}dl who showed that $p(16) > \frac{1}{4} (1 - \frac{1}{16})$ and $p(k) > \frac{1}{4} + \frac{1}{200}$ for all $k \geq 162$.  Given that the conjecture of Erd\H{o}s is true for $k \in \{2,3 \}$ but false for large $k$, it is natural to ask for the smallest value of $k$ for which $p(k) > \frac{1}{4} ( 1 - \frac{1}{k} )$.
In particular, the question of whether or not $p(4) = \frac{1}{4} ( 1 - \frac{1}{4} )$ was mentioned by Dudek and R\"{o}dl as an open problem.
We solve this problem by proving that
$p(4) \geq \frac{1}{4} (1 - \frac{1}{4} ) + \frac{1}{584064}$ and $p(k) > \frac{1}{4} (1 - \frac{1}{k})$ for $4 \leq k \leq 15$.  
We also show that $p(4) \leq \frac{1}{4}$ which improves upon the previously best known upper bound on $p(4)$.  
Therefore, $p(4)$ must lie somewhere between $\frac{3}{16} + \frac{1}{584064}$ and $\frac{1}{4}$.
\end{abstract}


\section{Introduction}

Let $G$ be an infinite graph with $V(G) = \{1,2,3, \dots \}$.
An \emph{increasing path of length} $k$, denoted $I_k$, is a sequence of $k+1$ vertices $i_1, \dots , i_{k+1}$ such that $i_1 < i_2 < \dots < i_{k+1}$ and $i_j$ is adjacent to $i_{j+1}$ for $1 \leq j \leq k$.
An infinite graph $G$ is \emph{$I_k$-free} if it contains no increasing path of length $k$.  For an infinite graph $G$, let $G_n$ be the subgraph of $G$ induced by the vertices $\{1, 2, \dots , n \}$
and $p(G) = \displaystyle\liminf_{n \rightarrow \infty} \frac{ e(G_n ) }{n^2}$.
Define the \emph{path Tur\'{a}n number of $I_k$}, denoted $p(k)$, to be the value
\begin{center}
$p(k) = \sup \{ p(G) : G ~ \textrm{is $I_k$-free} \}$.
\end{center}
Czipszer, Erd\H{o}s, and Hajnal \cite{CEH} introduced these path Tur\'{a}n numbers and proved the following.
\begin{theorem}[Czipszer, Erd\H{o}s, Hajnal \cite{CEH}]\label{ceh theorem}
The path Tur\'{a}n numbers $p(2)$ and $p(3)$ satisfy
\begin{center}
$p(2) = \frac{1}{8}$ and $p(3) = \frac{1}{6}$.
\end{center}
\end{theorem}
They also gave a simple construction that shows
\begin{center}
$p(k) \geq \frac{1}{4} \left( 1 - \frac{1}{k} \right)$ for all $k \geq 2$
\end{center}
and asked if $p(k) = \frac{1}{4} \left(1 - \frac{1}{k} \right)$ holds for $k \geq 4$.
Erd\H{o}s conjectured in \cite{erdos1} and \cite{erdos2} that $p(k) = \frac{1}{4} \left( 1 - \frac{1}{k} \right)$ holds for all $k \geq 2$.
In 2008, Dudek and R\"{o}dl \cite{DR} disproved the conjecture for certain values of $k$ by proving the
following result.
\begin{theorem}[Dudek, R\"{o}dl \cite{DR}]\label{dr lower bound}
The path Tur\'{a}n number $p(16)$ satisfies
\begin{equation*}
p(16) > \frac{1}{4} \left( 1 - \frac{1}{16} \right).
\end{equation*}
Furthermore,
\begin{equation*}
p(k) > \frac{1}{4} + \frac{1}{200}
\end{equation*}
for all $k \geq 162$.
\end{theorem}

The results of \cite{DR} and the conjecture $p(k) = \frac{1}{4} \left( 1 - \frac{1}{k} \right)$ is
mentioned in a survey paper of Komj\'{a}th \cite{komjath} which discusses some of the work of Erd\H{o}s in infinite graph theory.

Theorem \ref{ceh theorem} and Theorem \ref{dr lower bound}
suggest the following question: for which values of $k$ does one have
\begin{equation}\label{pk equation}
p(k)= \frac{1}{4} \left( 1 - \frac{1}{k} \right)
\end{equation}
and in particular, what is the smallest value of $k$ for which (\ref{pk equation}) holds?  Our first result is a construction
that shows (\ref{pk equation}) does not hold for several small values of $k$ and disproves the conjecture of Erd\H{o}s in the most difficult case; the case when $k = 4$.

\begin{theorem}\label{new lower bound}
If $4 \leq k \leq 15$, then
\[
p(k) > \frac{1}{4} \left( 1 - \frac{1}{k} \right).
\]
\end{theorem}

In light of Theorem \ref{new lower bound} and the results of \cite{DR}, it seems likely that (\ref{pk equation}) fails for all $k \geq 4$.

Using the argument of \cite{CEH}, we obtained the following upper bound on $p(4 )$.

\begin{theorem}\label{4 upper bound}
The path Tur\'{a}n number $p(4)$ satisfies
\[
p(4) \leq \frac{1}{4}
\]
\end{theorem}

By Theorems \ref{new lower bound} and \ref{4 upper bound}, we have
\begin{equation}\label{4 bounds}
\frac{1}{4} \left( 1 - \frac{1}{4} \right) + \frac{1}{584064} \leq p(4) \leq \frac{1}{4}.
\end{equation}
Determining the exact value of $p(4)$ is a challenging open problem.  Probably the lower bound in (\ref{4 bounds}) is closer to the
true value of $p(4)$.

The next section introduces a sequence reformulation of the path Tur\'{a}n problem.  This reformulation was a key ingredient in the
constructions of \cite{DR} and we use it in our constructions as well.  In
Section \ref{lower bound part 1} we give our construction method and state our main lemma.
Section \ref{proof of count 3} contains the proof of our main lemma.  In Section \ref{matrices} we prove
Theorem \ref{new lower bound} and in Section \ref{ub} we prove Theorem \ref{4 upper bound}.


\section{Sequences}\label{sequences}

It will be convenient to work with the sequence formulation of the problem introduced by Dudek and R\"{o}dl.
Given an $I_k$-free graph $G$ with $V(G) = \mathbb{N}$, partition $\nn$ into $k$ sets $N_1, \dots, N_k$ where
\begin{center}
$N_1 (G) = \{ n \in \nn : \forall m \in \nn ~ \textrm{with} ~ \{n,m \} \in E(G) ~ \textrm{we have} ~ n < m \}$
\end{center}
and for $2 \leq i \leq k$,

\begin{center}
$N_{i}(G) = \{ n \in \nn \backslash \displaystyle\bigcup_{j=1}^{i-1} N_{j} (G) : \forall m \in \nn ~ \textrm{with} ~ \{n,m\} \in E(G)$

~~~~~~~~~~~~~~~~~~~~~~~~~~~~~~~~~~~~~~~~~~ $ ~\textrm{we have} ~ n < m ~ \textrm{or} ~ m \in \displaystyle\bigcup_{j=1}^{i-1} N_{j}(G) \}$.
\end{center}

Define $C = C(G)$ to be the sequence $\{c_n \}_{n=1}^{ \infty}$ where $c_n = i $ if and only if $n \in N_i (G)$.
Let
\[
\mathcal{S}_C (n) = | \{ (i,j) : 1 \leq i < j \leq n \mbox{~and~} c_i < c_j \} |.
\]
It is shown in \cite{DR}, that
\[
\liminf_{n \rightarrow \infty} \dfrac{ e(G_n )}{n^2} = \liminf_{n \rightarrow \infty} \dfrac{ \mathcal{S}_C (n) }{n^2}.
\]

Conversely, given a sequence whose terms are elements of $[k]$, the corresponding infinite graph $G$ with
vertex set $\mathbb{N}$ has edge set
\[
\{ (i,j) : 1 \leq i < j ~ \mbox{and}~ c_i < c_j \}.
\]


\section{Proof of Theorem 1.3}

\subsection{Constructing Sequences}\label{lower bound part 1}

Let $k \geq 2$ and $l \geq 1$ be integers.  Let $D$ be a $k \times l$ matrix whose entries are non-negative integers.  Let
$d_{i,j}$ be the $(i,j)$-entry of $D$.  We will use $D$ to construct an infinite sequence $C$ with entries in $[k]$.  Let
$D_j$ be the sequence
\[
D_j = \underbrace{1 1  \cdots 1 1}_{d_{1,j} ~1's } \underbrace{2 2  \cdots 2 2}_{ d_{2,j} ~ 2's }
 \underbrace{3 3  \cdots 3 3}_{ d_{3,j} ~ 3's }  \cdots
\underbrace{ k k \cdots k k }_{d_{k,j} ~k's}.
\]
We call $D_j$ an \emph{atom}.
We remark that $D_j$ has length $\sum_{i=1}^{k} d_{i,j}$ and since the $d_{i,j}$'s can be zero, it is possible that $D_j$ does not contain every symbol from $[k]$.  Given a finite sequence $R$, let us write $L(R)$ for the length of $R$ so that, in this notation,
\[
L( D_j ) = \sum_{i=1}^k d_{i,j}.
\]
Given any two finite sequences $S = s_1 s_2 \dots s_x$ and $T = t_1 t_2 \dots t_y$, we write
\[
ST = s_1 s_2 \dots s_x t_1 t_2 \dots t_y
\]
for the concatenation of $S$ and $T$.
For an integer $m \geq 1$, define $B_m$ to be the sequence
\[
B_m = \underbrace{ D_1 D_1  \cdots D_1 D_1}_{2^{m-1} }
\underbrace{ D_2 D_2  \cdots D_2 D_2 }_{2^{m-1} }  \cdots
\underbrace{ D_l D_l  \cdots D_l D_l}_{2^{m-1} }.
\]
We call the sequence $B_m$ a \emph{block}.
Since $B_m$ contains $2^{m-1}$ copies of $D_j$ for $1 \leq j \leq l$, the length of $B_m$ is
\[
L(B_m) = 2^{m-1} \sum_{j=1}^{l} L(D_j) = 2^{m-1} \sum_{j=1}^{l} \sum_{i=1}^{k} d_{i,j}.
\]
Define $C=C(D)$ to be the infinite sequence
\[
C = B_1 B_2 B_3 B_4 B_5 \cdots.
\]
Write $C = \{ c_r \}_{r=1}^{ \infty}$ for this sequence.
For example, if $D = \begin{pmatrix} 1 & 3 \\ 0 & 2 \\ 2 & 1 \end{pmatrix}$, then $D_1 = 133$, $D_2 = 111223$, and
\[
C = \underbrace{133111223}_{B_1} \underbrace{133133111223111223}_{B_2} \cdots
\]

Motivated by the infinite $I_{k}$-free graph corresponding to the sequence $C$, we call a pair
$\{ i,j \}$ for which $1 \leq i < j$ and $c_i < c_j$ an \emph{edge}.
Let $M$ be the $l \times l$ matrix whose $(i,j)$-entry, denoted by $m_{i,j}$, is given by
\[
m_{i,j} = \sum_{x = 1}^{k-1} d_{x,i} \left( \sum_{y = x + 1}^k d_{y,j} \right).
\]
The value $m_{i,j}$ is the number of edges with one endpoint in an atom $D_i$, the other endpoint in an atom $D_j$,
and where $D_i$ precedes $D_j$ in the sequence $C$.  In general, the matrix $M$ is not a symmetric matrix.
Define
\[
w_1 (M) = \sum_{i=1}^{l} \sum_{j=1}^l m_{i,j}, ~~ w_2 (M) = \sum_{i=1}^{l-1} \sum_{j=i+1}^{l} m_{i,j},~~
w_3 (M) = \sum_{i=1}^{l} m_{i,i},
\]
and $w (M) = \frac{ w_1 (M) }{3} + \frac{ w_2 (M) }{3} + \frac{ w_3 (M) }{6}$.

We want to choose $D$ so that for the corresponding sequence $C=C(D)$,
\begin{equation}\label{lim inf}
\liminf_{n \rightarrow \infty} \dfrac{ \mathcal{S}_C (n) }{n^2}
\end{equation}
is as large as possible.  Our main tool for estimating (\ref{lim inf}) is the following lemma.

\begin{lemma}\label{count 3}
Given $D$, $C$, and $M$ as above, the value of $\displaystyle\liminf_{n \rightarrow \infty} \frac{ \mathcal{S}_C (n) }{n^2}$
is at least the minimum value of
\[
\frac{ w(M) + \displaystyle\sum_{j=1}^{t-1} \left( \sum_{i=1}^l m_{i,j} + \frac{m_{j,j}}{2} + \sum_{i=1}^{j-1} m_{i,j} \right)
+ \epsilon  \sum_{i=1}^l m_{i,t} + \epsilon^2 \frac{  m_{t,t} }{2} + \epsilon \sum_{i=1}^{t-1} m_{i,t}  }
{ \left(
\displaystyle\sum_{i=1}^k \sum_{j=1}^{l} d_{i,j} + \sum_{j = 1}^{t-1} \sum_{i=1}^k d_{i,j} + \epsilon
\sum_{i=1}^k d_{i,t} \right)^2 }
\]
where $t$ ranges over all integers in $\{ 1,2, \dots , l \}$ and $\epsilon$ ranges over all real numbers in the interval $[0,1]$.

\end{lemma}

\noindent
\textbf{Remark:} In Lemma \ref{count 3} and in the rest of the paper, any sum of the form $\sum_{i=1}^{t-1} \alpha_i$ with $t = 1$ is taken to be 0.


\subsection{Proof of Lemma 3.1}\label{proof of count 3}

Let $n$ be a positive integer.  We choose $m$ to be the largest integer such that
\begin{equation}\label{def of n}
 \sum_{x = 1}^{m} \left( 2^{x-1} \sum_{j=1}^l \sum_{i=1}^k d_{i,j} \right) \leq n <
\sum_{x=1}^{m+1} \left( 2^{x-1} \sum_{j=1}^l \sum_{i=1}^k d_{i,j} \right).
\end{equation}
The left hand side of (\ref{def of n}) is $ L(B_1 B_2 \cdots B_m)$. The right hand side is $L(B_1 B_2 \cdots B_m B_{m+1} )$.
We can write $n$ in the form
\[
n = L(B_1 B_2 \cdots B_m) + L(\underbrace{D_1 D_1 \cdots D_1}_{2^m}) +
 \dots
+ L(\underbrace{D_{t-1} D_{t-1} \cdots D_{t-1}}_{2^m})
 + \epsilon
L(\underbrace{D_t D_t \cdots D_t}_{2^m})
\]
for some $0 \leq \epsilon \leq 1$ and $t \in \{1,2, \dots , l \}$.  Therefore,
\begin{eqnarray*}
n & = & \sum_{x=1}^m \left( 2^{x-1} \sum_{i=1}^k \sum_{j=1}^l  d_{i,j} \right) +
2^{m} \left(    \sum_{i=1}^k d_{i,1} + \dots +    \sum_{i=1}^k d_{i,t-1} +
\epsilon    \sum_{i=1}^k d_{i,t} \right)   \\
& = & 2^m \left(  (1 - \frac{1}{2^m} )  \left( \sum_{i=1}^k \sum_{j=1}^l  d_{i,j} \right)  +
\left(   \sum_{x=1}^{t-1}  \sum_{i=1}^k d_{i,x}  +
\epsilon    \sum_{i=1}^k d_{i,t} \right) \right).
\end{eqnarray*}
This is the formula that we will use for $n$ in the expression
$\frac{ \mathcal{S}_C (n) }{n^2}$.  It is helpful to think of $n$ as the location of a cut in the sequence $C$.  When we look
at the first $n$ terms of the sequence, we see all of the terms in blocks $B_1$ through $B_m$, and only some of the terms in the
block $B_{m+1}$.  For this reason, we call $B_{m+1}$ a \emph{partial block}.  Clearly the
number of elements that we see from $B_{m+1}$ depends on
$n$.

Next we look for a lower bound on
$\mathcal{S}_C (n) $.

\begin{lemma}\label{count 1}
For $1 \leq k_1 < k_2$, the number of edges between the block $B_{k_1}$ and the block $B_{k_2}$ is
\[
\frac{2^{k_1+k_2} }{4} \sum_{i=1}^l \sum_{j=1}^l m_{i,j}.
\]
\end{lemma}
\begin{proof}
The sequence $B_{k_1}$ contains $2^{k_1-1}$ atoms of type $D_j$ for each
$j \in \{1,2 , \dots , l \}$.  A similar assertion holds for $B_{k_2}$.
There are $2^{k_1 - 1}2^{k_2 - 1} m_{i,j}$ edges from the
$D_i$ atoms in block $B_{k_1}$ to the $D_j$ atoms in block in $B_{k_2}$.  The proof of the lemma is completed by summing over all $i,j$ with
$1 \leq i , j \leq l$.
\end{proof}

\begin{lemma}\label{count 2}
For any $k \geq 1$, the number of edges in $B_k$ is at least
\[
\frac{ 4^k}{4} \sum_{i=1}^{l-1} \sum_{j=i+1}^l m_{i,j} + \frac{4^k}{8} \sum_{i=1}^l m_{i,i} - c_D2^k
\]
where $c_D$ is a constant that depends only on $D$.
\end{lemma}
\begin{proof}
For any $1 \leq i < j \leq l$, the block $B_k$ contains $2^{k - 1}$ consecutive atoms of type $D_i$ that
precede $2^{k-1}$ consecutive atoms of type $D_j$.  Therefore, $B_k$ contains at least
\[
2^{k - 1} 2^{k-1} \sum_{i=1}^{l-1} \sum_{j=i+1}^l m_{i,j}
\]
edges that have end points in atoms of different types.
Next we count edges that have both endpoints in an atom of type $D_i$.  There are $\binom{2^{k-1} }{2}$ pairs of distinct
atoms of type $D_i$ in the block $B_k$ and a total of $m_{i,i}$ edges between any two such atoms.
 Summing over $1 \leq i \leq l$ gives a total of
\[
\binom{ 2^{k-1} }{2} \sum_{i=1}^{l} m_{i,i} =
\frac{4^k}{8} \sum_{i=1}^l m_{i,i} - c_D2^k
\]
edges.
\end{proof}

\bigskip

A consequence of Lemmas \ref{count 1} and \ref{count 2} is that the number of edges contained in $B_1 B_2 \cdots B_m$ is at least
\begin{equation*}
\frac{1}{4} w_1(M) \sum_{k_1 = 1}^{m-1}  \sum_{k_2 = k_1 +1}^m 2^{ k_1 + k_2 }
+ \left( \frac{1}{4} w_2(M) + \frac{1}{8} w_3(M) \right)
\left( \sum_{k=1}^{m} 4^k \right) - c_D 2^{m+1}
\end{equation*}
A short calculation shows that this expression can be simplified to
\begin{equation*}
4^m \left( \frac{ w_1 (M) }{3} + \frac{ w_2 (M) }{3} + \frac{ w_3 (M) }{6} \right) - O (2^m)
\end{equation*}
where the constant in the $O$ notation only depends on $D$.
Since $w(M )=  \frac{ w_1 (M) }{3} + \frac{ w_2 (M) }{3} + \frac{ w_3 (M) }{6}$,
we have the lower bound
\[
 \mathcal{S}_C (n)   \geq 4^m
\left(   \frac{ w_1 (M) }{3} + \frac{ w_2 (M) }{3} + \frac{ w_3 (M) }{6} \right) - O (2^m)  = 4^m  w(M) - O (2^m)
\]
however this is not good enough, especially in the case when $t$ is close to $l$ which, in terms of $n$,
means that $n$ is closer to $L(B_1 B_2 \cdots B_{m+1})$ than
it is to $L(B_1 B_2 \cdots B_m)$.
We are losing too much by not counting edges between
$B_1 B_2 \cdots B_m$ and the $D_i$'s coming from the partial block $B_{m+1}$, as well as the edges in the
partial block $B_{m+1}$.  To count these edges we need a few more lemmas.

\begin{lemma}\label{more count 1}
The number of edges with one endpoint in $B_1B_2 \cdots B_m$ and the other endpoint in an atom of type $D_j$ in the partial block $B_{m+1}$ where $j \in \{1,2, \dots , t \}$ is
\[
4^m \left( 1 - \frac{1}{2^m } \right) \left( \sum_{i=1}^{l} \sum_{j=1}^{t-1} m_{i,j} + \epsilon \sum_{i=1}^{l} m_{i,t} \right).
\]
\end{lemma}
\begin{proof}
For any $i \in \{1,2, \dots , l \}$, there are $2^m - 1$ atoms of type $D_i$ in the sequence $B_1 B_2 \cdots B_m$.  Each such atom
sends $m_{i,j}$ edges to a $D_j$ in the partial block $B_{m+1}$.  For $j \in \{ 1 , 2 \dots , t - 1 \}$, the partial block $B_{m+1}$
contains $2^m$ atoms of type $D_j$.  The partial block $B_{m+1}$ contains $\epsilon 2^m$ atoms of type $D_t$.  Summing over $1 \leq j \leq t - 1$ gives a total of
\[
(2^m - 1) 2^m \sum_{i=1}^{l} \sum_{j=1}^{t - 1} m_{i,j} + (2^m  - 1) \epsilon 2^m \sum_{i=1}^{l} m_{i, t }
\]
edges.
\end{proof}

\begin{lemma}\label{more count 2}
The number of edges both of whose endpoints are contained in the partial block $B_{m+1}$ is at least
\begin{equation}\label{new eq}
\frac{4^m}{2} \left( 1 - \frac{1}{2^m} \right) \sum_{j=1}^{t-1} m_{j,j} + 4^m \sum_{i=1}^{t-2} \sum_{j=i+1}^{t-1} m_{i,j} +
\epsilon^2 \frac{4^m}{2} \left( 1 - \frac{1}{2^m \epsilon } \right) m_{t,t} + \epsilon 4^m \sum_{i=1}^{t-1} m_{i,t}.
\end{equation}
\end{lemma}
\begin{proof}
The proof of the lemma is similar to the proofs of the previous lemmas.  Instead of going through the details, we simply state what types of edges each of the four terms in (\ref{new eq}) is counting.  The sum
\[
\frac{4^m}{2} \left( 1 - \frac{1}{2^m} \right) \sum_{j=1}^{t-1} m_{j,j}
\]
counts edges both of whose endpoints are in an atom of type $D_j$ in the partial block $B_{m+1}$ for $j \in \{1,2, \dots , t - 1 \}$.
There are $\binom{2^m}{2}$ pairs of such $D_j$.

The second term counts edges in the partial block $B_{m+1}$ where one endpoint is an atom of type $D_i$ and the other
endpoint is an atom of type $D_j$ where $1 \leq i < j \leq t - 1$.

The third term
\[
\epsilon^2 \frac{4^m}{2} \left( 1 - \frac{1}{2^m \epsilon } \right) m_{t,t}
\]
counts edges whose endpoints are in an atom of type $D_t$.  There are $\binom{ \epsilon 2^m }{2}$ pairs of distinct atoms of type $D_t$ in the partial block $B_{m+1}$.

The final term counts edges with one endpoint in an atom of type $D_i$ where $1 \leq i < t$, and the other endpoint is in an atom of type $D_t$.  There are $2^m \epsilon 2^m$ such pairs and we then sum this over $1 \leq i \leq t - 1$.

\end{proof}

From Lemmas \ref{more count 1} and \ref{more count 2}, we now have
\begin{eqnarray*}
\mathcal{S}_C (n) & \geq & 4^m \left( w(M)
+ \sum_{j=1}^{t-1} \left( (1 - 1/2^m ) \sum_{i=1}^{l} m_{i,j} +
(1 - 1/2^m) \frac{ m_{j,j} }{2} + \sum_{i=1}^{j-1} m_{i,j} \right) \right)  \\
& + & 4^m \left( (1-1/2^m) \epsilon \sum_{i=1}^{l} m_{i,t} + \epsilon^2 \frac{1}{2}
\left( 1 - \frac{1}{2^m \epsilon } \right) m_{t,t} +
\epsilon  \sum_{i=1}^{t-1} m_{i,t} \right) - O(2^m)
\end{eqnarray*}

Now as $n$ goes to infinity, $m$ must also tends to infinity.  Combining this lower
bound on $\mathcal{S}_C (n)$ together with
\[
n = 2^m \left(  (1 - \frac{1}{2^m} )  \left( \sum_{i=1}^k \sum_{j=1}^l  d_{i,j} \right)  +
\left(   \sum_{j=1}^{t-1}  \sum_{i=1}^k d_{i,j}  +
\epsilon    \sum_{i=1}^k d_{i,t} \right) \right)
\]
completes the proof of Lemma \ref{count 3}.


\subsection{Choosing Matrices}\label{matrices}

In this section we give several matrices which, when combined with Lemma \ref{count 3}, improve the lower bound
\[
p(k ) \geq \frac{1}{4} \left(1 - \frac{1}{k} \right)
\]
for different values of $k$.  We first list the matrices and then give the corresponding lower bounds obtained from
Lemma \ref{count 3}.  Computations were done using Mathematica \cite{mathematica} and the code used for the computations is given in the Appendix.

\begin{center}
\textbf{The matrices used to improve $p(k) \geq \frac{1}{4} \left( 1 - \frac{1}{k} \right)$}.
\end{center}

\begin{center}
$D(5) = \left(
\begin{array}{ccc}
 6 & 2 & 1 \\
 1 & 7 & 3 \\
 2 & 2 & 8 \\
 6 & 2 & 2 \\
 4 & 5 & 2 \\
\end{array}
\right)$ ~~
$D(6) = \left(
\begin{array}{ccc}
 6 & 0 & 2 \\
 3 & 6 & 0 \\
 2 & 3 & 5 \\
 2 & 5 & 3 \\
 3 & 2 & 4 \\
 6 & 2 & 0 \\
\end{array}
\right)$ ~~
$D(7) = \left(
\begin{array}{cccc}
 3 & 0 & 2 & 3 \\
 3 & 4 & 2 & 0 \\
 1 & 2 & 5 & 1 \\
 0 & 2 & 4 & 3 \\
 2 & 0 & 4 & 3 \\
 3 & 1 & 3 & 4 \\
 3 & 3 & 2 & 0 \\
\end{array}
\right)$
\end{center}

\begin{center}
$D(8) =\left(
\begin{array}{cccc}
 5 & 2 & 0 & 0 \\
 0 & 5 & 2 & 2 \\
 0 & 2 & 7 & 1 \\
 0 & 1 & 2 & 7 \\
 1 & 1 & 4 & 3 \\
 2 & 0 & 1 & 6 \\
 4 & 2 & 2 & 0 \\
 1 & 5 & 2 & 0 \\
\end{array}
\right)$
~
$ D(9) = \left(
\begin{array}{ccccc}
 7 & 0 & 0 & 1 & 2 \\
 0 & 6 & 0 & 1 & 3 \\
 1 & 1 & 6 & 2 & 1 \\
 1 & 2 & 2 & 6 & 2 \\
 1 & 1 & 1 & 0 & 8 \\
 2 & 2 & 6 & 2 & 1 \\
 0 & 1 & 2 & 6 & 0 \\
 3 & 2 & 0 & 1 & 7 \\
 5 & 3 & 1 & 0 & 0 \\
\end{array}
\right)$
\end{center}

\begin{center}
$ D(10) =
\left(
\begin{array}{ccccc}
 10 & 7 & 2 & 5 & 8 \\
 6 & 10 & 11 & 4 & 3 \\
 4 & 7 & 11 & 9 & 6 \\
 4 & 3 & 9 & 14 & 9 \\
 9 & 3 & 2 & 9 & 10 \\
 10 & 9 & 6 & 3 & 7 \\
 9 & 11 & 6 & 4 & 4 \\
 6 & 9 & 12 & 8 & 1 \\
 7 & 5 & 7 & 9 & 8 \\
 7 & 7 & 6 & 6 & 9 \\
\end{array}
\right)$~~~
$D(11) =
\left(
\begin{array}{cccccc}
 7 & 1 & 0 & 1 & 2 & 1 \\
 0 & 7 & 1 & 0 & 0 & 0 \\
 0 & 1 & 6 & 0 & 0 & 2 \\
 1 & 0 & 0 & 6 & 0 & 2 \\
 1 & 1 & 0 & 1 & 6 & 1 \\
 0 & 0 & 2 & 1 & 1 & 6 \\
 1 & 1 & 0 & 6 & 0 & 2 \\
 1 & 1 & 0 & 2 & 6 & 2 \\
 0 & 1 & 2 & 2 & 1 & 5 \\
 7 & 0 & 1 & 0 & 0 & 0 \\
 2 & 6 & 2 & 1 & 0 & 0 \\
\end{array}
\right)$
\end{center}

\begin{center}
$D(12) =
\left(
\begin{array}{cccccc}
 6 & 1 & 0 & 0 & 0 & 0 \\
 0 & 6 & 1 & 0 & 0 & 0 \\
 0 & 1 & 7 & 0 & 0 & 1 \\
 0 & 0 & 1 & 6 & 1 & 0 \\
 0 & 1 & 0 & 1 & 6 & 1 \\
 0 & 0 & 1 & 1 & 0 & 6 \\
 1 & 0 & 0 & 6 & 1 & 1 \\
 1 & 0 & 1 & 1 & 6 & 0 \\
 1 & 0 & 1 & 0 & 1 & 6 \\
 6 & 0 & 1 & 1 & 1 & 0 \\
 1 & 7 & 1 & 0 & 0 & 0 \\
 1 & 1 & 6 & 0 & 0 & 0 \\
\end{array}
\right)
$
~~~
$D(13) =
\left(
\begin{array}{ccccccc}
 7 & 0 & 0 & 0 & 0 & 1 & 0 \\
 1 & 7 & 0 & 0 & 0 & 0 & 1 \\
 1 & 1 & 8 & 0 & 0 & 1 & 1 \\
 0 & 1 & 0 & 7 & 0 & 1 & 0 \\
 0 & 0 & 0 & 0 & 8 & 1 & 1 \\
 1 & 0 & 1 & 1 & 1 & 7 & 1 \\
 0 & 1 & 0 & 1 & 0 & 1 & 8 \\
 1 & 1 & 0 & 0 & 8 & 0 & 0 \\
 1 & 0 & 1 & 1 & 1 & 7 & 1 \\
 0 & 0 & 0 & 1 & 0 & 1 & 8 \\
 8 & 0 & 1 & 0 & 0 & 0 & 0 \\
 1 & 7 & 0 & 0 & 0 & 1 & 1 \\
 1 & 1 & 7 & 0 & 1 & 1 & 1 \\
\end{array}
\right)
$
\end{center}

\begin{center}
$D(14) =
\left(
\begin{array}{ccccccc}
 7 & 1 & 1 & 0 & 0 & 0 & 0 \\
 1 & 7 & 0 & 1 & 1 & 1 & 0 \\
 0 & 0 & 7 & 0 & 0 & 1 & 0 \\
 1 & 1 & 0 & 7 & 1 & 0 & 1 \\
 0 & 1 & 1 & 1 & 7 & 1 & 0 \\
 1 & 0 & 1 & 0 & 0 & 8 & 1 \\
 0 & 1 & 0 & 0 & 0 & 1 & 8 \\
 1 & 0 & 1 & 8 & 0 & 1 & 1 \\
 0 & 1 & 1 & 0 & 7 & 1 & 0 \\
 1 & 0 & 0 & 0 & 1 & 7 & 0 \\
 1 & 0 & 1 & 0 & 0 & 0 & 8 \\
 8 & 0 & 0 & 1 & 1 & 0 & 0 \\
 0 & 8 & 0 & 1 & 0 & 0 & 1 \\
 1 & 1 & 8 & 1 & 1 & 0 & 0 \\
\end{array}
\right)
$
~~~
$D(15) =
\left(
\begin{array}{cccccccc}
 9 & 0 & 0 & 0 & 0 & 0 & 0 & 0 \\
 0 & 8 & 0 & 0 & 0 & 0 & 1 & 1 \\
 0 & 1 & 9 & 0 & 0 & 0 & 0 & 0 \\
 0 & 1 & 0 & 9 & 0 & 1 & 0 & 1 \\
 0 & 1 & 0 & 0 & 9 & 0 & 1 & 0 \\
 0 & 0 & 1 & 0 & 1 & 9 & 1 & 1 \\
 0 & 0 & 1 & 1 & 0 & 1 & 9 & 1 \\
 0 & 1 & 1 & 1 & 1 & 0 & 1 & 8 \\
 0 & 0 & 1 & 0 & 9 & 1 & 0 & 0 \\
 0 & 0 & 1 & 0 & 1 & 9 & 0 & 0 \\
 1 & 1 & 1 & 0 & 1 & 0 & 9 & 0 \\
 1 & 1 & 0 & 0 & 0 & 1 & 0 & 8 \\
 9 & 1 & 1 & 0 & 0 & 0 & 0 & 1 \\
 0 & 8 & 0 & 0 & 1 & 0 & 0 & 1 \\
 1 & 0 & 9 & 1 & 1 & 1 & 1 & 0 \\
\end{array}
\right)
$
\end{center}

\begin{center}
\textbf{Corresponding lower bounds}
\end{center}

\begin{center}
\begin{tabular}{ll}
$p(5) \geq \frac{1688}{8427} = \frac{1}{4} \left( 1 - \frac{1}{5} \right) + \frac{13}{42135}$ ~~     &
$p(6) \geq \frac{3683}{17672} = \frac{1}{4} \left( 1 - \frac{1}{6} \right) + \frac{1}{13254}$  ~~    \medskip             \\
$p(7) \geq \frac{365}{1701} = \frac{1}{4} \left( 1 - \frac{1}{7} \right) + \frac{1}{3402}$ ~~       &
$p(8) \geq \frac{19325}{87846} = \frac{1}{4} \left( 1 - \frac{1}{8} \right) + \frac{ 1739}{ 1405536}$ ~~ \medskip  \\
$p(9) \geq \frac{ 9448}{42483} = \frac{1}{4} \left( 1 - \frac{1}{9} \right) + \frac{ 22}{127449}$ ~~ &
$p(10) \geq \frac{ 83234}{369603} = \frac{1}{4} \left( 1 - \frac{1}{10} \right) + \frac{ 2933}{ 14784120}$ ~~ \medskip \\
$p(11) \geq \frac{ 18033}{ 79202} = \frac{1}{4} \left( 1 - \frac{1}{11} \right) + \frac{ 179}{ 435611}$ ~~ &
$p(12) \geq \frac{ 13511}{58482} = \frac{1}{4} \left( 1 - \frac{1}{12} \right) + \frac{ 871}{ 467856}$ ~~ \medskip \\
$p(13) \geq \frac{ 57931}{ 249696} = \frac{1}{4} \left( 1 - \frac{1}{13} \right) + \frac{ 4015}{ 3246048}$ ~~ &
$p(14) \geq \frac{ 16743}{71824} = \frac{1}{4} \left( 1 - \frac{1}{14} \right) + \frac{ 487}{ 502768}$ ~~ \medskip \\
$p(15) \geq \frac{ 36251}{154568} = \frac{1}{4} \left( 1 - \frac{1}{15} \right) + \frac{ 2777}{ 2318520}$ ~~ & \\

\end{tabular}

\end{center}
We take a moment to briefly describe the method in which these matrices were obtained. The matrix $D(5)$ was obtained by starting with the matrix
\[
R(5) := \begin{pmatrix} 3 & 0 & 0 \\ 0 & 3 & 0 \\ 0 & 0 & 3 \\ 3 & 0 & 0 \\ 0 & 3 & 0 \end{pmatrix}.
\]
The idea to use this matrix
as a starting point comes from the fact that the constructions in \cite{DR} are good for large values of $k$ and so, while they do not improve the lower bound $p(k) \geq \frac{1}{4} (1 - \frac{1}{k} )$ for small $k$, they still give a reasonable bound for small $k$.
The matrix $R(5)$ is a natural modification of the construction in \cite{DR}.
We then added a random 0-1 matrix to $R(5)$.
This was repeated many times until we found a new matrix that provided a better lower bound on $p(5)$ than the lower bound
given by $R(5)$.  This process was repeated until we arrived at the matrix $D(5)$ given above.  Looking closely at
each of the matrices above, one may be able to find the ``dominant diagonal" entries.  For instance in $D(8)$,
the positions
\[
(1,1), (2,2), (3,3), (4,4), (5,3),(6,4),(7,1),(8,2)
\]
contain entries that are larger than the other entries.  The initial matrix used to construct $D(8)$ is
\[
R(8):= \begin{pmatrix} 4 & 0 & 0 & 0 \\ 0 & 4 & 0 & 0 \\ 0 & 0 & 4 & 0 \\ 0 & 0 & 0 & 4 \\
0 & 0 & 4 & 0 \\ 0 & 0 & 0 & 4 \\ 4 & 0 & 0 & 0 \\ 0 & 4 & 0 & 0
\end{pmatrix}
\]

In the case that $k = 4$, the matrix
\[
D(4) = \left(
\begin{array}{cccccccc}
 6 & 2 & 5 & 8 & 4 & 6 & 6 & 9 \\
 7 & 5 & 7 & 6 & 5 & 5 & 4 & 4 \\
 4 & 5 & 7 & 6 & 5 & 8 & 4 & 7 \\
 7 & 2 & 5 & 8 & 4 & 5 & 6 & 8 \\
\end{array}
\right)
\]
was obtained by a more ad hoc method.  It leads to the lower bound
\[
p(4) \geq \frac{ 109513}{584064} = \frac{1}{4} \left( 1 - \frac{1}{4} \right)+ \frac{1}{584064}.
\]
The corresponding sequence on 4 symbols has 8 atoms.


\section{Proof of Theorem 1.4}\label{ub}

We follow the method of \cite{CEH}.
Let $G$ be an infinite graph that is $I_4$-free.
Let $C = C(G)= \{c_n \}$ be its associated sequence on the symbols $\{1,2,3,4 \}$.  Define three sequences $\{u_n \}$, $\{v_n \}$, and $\{ w_n \}$ as follows.

\begin{enumerate}

\item For $i \in \mathbb{N}$, $u_i = k$ if and only if $c_k \in \{2,3,4 \}$ and $| \{ r : c_r \in \{2,3,4 \} , r \leq k \} | = i$.

\item For $i \in \mathbb{N}$, $v_i = k$ if and only if $c_{u_k} \in \{3,4 \}$ and
$|  \{ r : c_r \in \{3,4 \} , r \leq u_k \}  | = i$.

\item For $i \in \mathbb{N}$, $w_i = k$ if and only if $c_{u_{v_k} }  \in \{4 \}$ and
$| \{ r : c_r \in \{ 4 \} , r \leq u_{v_k} \}  | = i$.
\end{enumerate}

We give the following example for convenience.  In several of our counting arguments it may be quite useful for the reader to refer back to this example.

\begin{center}
\begin{tabular}{|c|c|c|c|c|c|c|c|c|c|c|} \hline
$n$      & 1  & 2              & 3 &  4                & 5                 & 6               & 7 & 8                & 9               & 10                    \\ \hline
$c_n$   & 1 & 2              & 1  & 3                & 2                 & 4               & 1 & 4                & 2                & 3                      \\ \hline
$u_i$          & ~ & $u_1=2$ & ~ & $u_2 = 4$ & $u_3  = 5$ & $u_4 = 6$ & ~ & $u_5 = 8$ & $u_6 = 9$ & $u_7 = 10$  \\ \hline
$v_i$          & ~ &  ~            & ~ & $v_1 = 2$  & ~               & $v_2 = 4$ & ~ & $v_3 = 5$ & ~               & $v_4 = 7$    \\
\hline
$w_i$         &  ~ & ~             & ~ & ~               & ~               &  $w_1 = 2$ & ~ & $w_2 = 3$ & ~             & ~                 \\ \hline
\end{tabular}
\smallskip

\textbf{Example 1}
\end{center}

Given $n \in \mathbb{N}$, we call a 4-tuple of positive integers $(n,j,k,l)$ a \emph{cut} if
\begin{center}
$u_j \leq n < u_{j+1}$,~ $v_k \leq j < v_{k+1}$, ~and~ $w_l \leq k < w_{l+1}$.
\end{center}
For instance, $(9,6,3,2)$ and $(10,7,4,2)$ are cuts for the sequence in Example 1.

If $n \in \mathbb{N}$ and $(n,j,k,l)$ is a cut, then
\begin{equation}\label{first ub}
e(G_n) \leq \sum_{i=1}^{j} (u_i - i) + \sum_{i=1}^{k} (v_i - i) + \sum_{i=1}^{l} (w_i - i ).
\end{equation}
Indeed, the sum $\sum_{i=1}^{j} (u_i - i)$ counts all pairs of the form $(c_s , c_t)$ with
$1 \leq s < t \leq n$, $c_s = 1$, and $c_t \in \{2,3,4 \}$.  Similarly, the sum
$\sum_{i=1}^k (v_i - i)$ counts all pairs of the form $(c_s , c_t)$ with
$1 \leq s < t \leq n$, $c_s = 2$, and $c_t \in \{3,4 \}$.  The
sum $\sum_{i=1}^l ( w_i - i )$ counts all pairs of the form
$(c_s , c_t)$ with $1 \leq s < t \leq n$, $c_s = 3$, and $c_t=4$.

Given $S \subseteq \{1,2,3,4 \}$, define
\[
\alpha_S ( i) =  | \{ r : c_r \in S ~ \mbox{and}~ r < i \} |.
\]
We claim that if $(n,j,k,l)$ is a cut, then
\begin{equation}\label{double count 1}
jk = \sum_{i=1}^k v_i + \sum_{i=1}^j \alpha_{ \{3,4 \} } (u_i )
\end{equation}
and
\begin{equation}\label{double count 2}
kl = \sum_{i=1}^l w_i + \sum_{i=1}^k \alpha_{ \{4 \} } ( u_{v_i} ).
\end{equation}
To prove these equalities, we will count pairs of the form $(u_s , v_t)$ with $1 \leq s \leq j$ and $1 \leq t \leq k$, as well as
pairs of the form $(v_s , w_t)$ with $1 \leq s \leq k$ and $1 \leq t \leq l$.

Clearly there are $jk$ pairs $(u_s , v_t)$ with $1 \leq s \leq j$ and $1 \leq t \leq k$.  The sum
$\sum_{i=1}^k v_i$ counts all pairs $(u_s , v_t)$ for which $s \leq v_t$ while the sum
$\sum_{i=1}^j \alpha_{ \{3,4 \} } (u_i)$ counts all pairs $(u_s , v_t)$ for which $s  > v_t$.  This double counting is best illustrated by referring to Example 1.  In terms of Example 1, $v_t$ is precisely the number of $u_s$'s for $s \leq v_t$ so that
the sum $\sum_{i=1}^k v_i$ counts pairs where the $u_s$ is directly above or to the left of $v_t$.  The sum
$\sum_{i=1}^j \alpha_{ \{ 3,4 \} } (u_i)$ then counts all pairs where the $u_s$ is to the right of $v_t$.  This shows that
(\ref{double count 1}) holds and a similar argument gives (\ref{double count 2}).

Combining (\ref{first ub}), (\ref{double count 1}), and (\ref{double count 2}) we have
\begin{equation}\label{second ub}
e(G_n) \leq \sum_{i=1}^j ( u_i - \alpha_{ \{3,4 \} } (u_i ) ) - \sum_{i=1}^k \alpha_{ \{4 \} } (u_{v_i} ) + jk + kl
- \frac{1}{2} ( (j+1)^2 + (k+1)^2 + (l + 1)^2 )
\end{equation}
for any cut $(n,j,k,l)$.

To estimate the first sum, we use the following lemma.
\begin{lemma}[Czipszer, Erd\H{o}s, Hajnal \cite{CEH}]\label{lemma}
Let $s_n, t_n$ be nondecreasing sequences of natural numbers such that $s_n - t_n > 0$ for all $n \in \nn$.  There exists a sequence $n_1 < n_2 < \dots $ such that for all $r$,
\begin{equation*}
\sum_{i=1}^{n_r} (s_{i} - t_{i} ) \leq \frac{1}{2} n_r ( s_{n_r} - t_{n_r} ) + o ( s_{n_r}^{2} ).\label{lemma estimate}
\end{equation*}
\end{lemma}

We apply Lemma~\ref{lemma} to the sequence
\[
\{ u_j - \alpha_{ \{3,4 \} } (u_j) \}_{j=1}^{ \infty}
\]
to obtain a sequence $j_1 < j_2 < \dots $ such that
\begin{equation}\label{estimate 2}
\sum_{i=1}^{ j_r } (u_i - \alpha_{ \{3,4 \} }  (u_i) ) \leq \frac{1}{2} j_r  ( u_{ j_r } - \alpha_{ \{3,4 \} } ( u_{j_r} ) ) + o ( u_{j_r }^{2} )
\end{equation}
for all $r \in \mathbb{N}$.  For $r \in \mathbb{N}$, define sequences $n_r$, $k_r$, and $l_r$, in terms of $j_r$, by
\begin{enumerate}
\item $n_r = u_{j_r}$,
\item $k_r$ is the largest index for which $v_{k_r} \leq  j_r$, and
\item $l_r$ is the largest index for which $w_{l_r} \leq k_r$.
\end{enumerate}
We then consider the sequence $\{ ( n_r , j_r , k_r , l_r ) \}_{r = 1}^{ \infty}$ of cuts.  By considering these sequence of
cuts, we are now looking at the subgraphs $G_{n_1}$, $G_{n_2} , \dots$ and for these subgraphs, we know that
(\ref{estimate 2}) holds.

For any $r$, we have $u_{j_r} \geq j_r \geq k_r \geq l_r$ so that
\begin{equation}\label{estimate 3}
j_r = o( u_{j_r}^2 ) ~~ \mbox{and} ~~k_r = o (u_{j_r}^2 ) ~~ \mbox{and} ~~ l_r = o( u_{j_r}^2 )
\end{equation}
as $u_{j_r} \rightarrow \infty$.
Furthermore,
\begin{equation}\label{estimate 4}
\sum_{i=1}^{k_r} \alpha_{ \{ 4 \} } (   u_{v_i} ) \geq ( l_r - 1) + (l_r - 2) + \dots + 2 + 1.
\end{equation}
In terms of Example 1, the sum on the left hand side of (\ref{estimate 4}) moves across the entries in the $v_i$ row, and
counts 4's that are above and to the left of the current entry.
A simple calculation shows that one would have
$\sum_{i=1}^{4} \alpha_{ \{4 \} } (u_{v_i} )  = 0 + 0 + 1 + 2 $ for Example 1.
The sum $\sum_{i=1}^{k_r} \alpha_{ \{ 4 \} } ( u_{v_i } )$ is minimized when all of the
$l_r$ 4's that appear in the cut $(n_r , j_r , k_r , l_r )$ come after all of the 3's in the cut.
Lastly, for any $r$,
\begin{equation}\label{estimate 5}
\alpha_{ \{3,4 \} } ( u_{ j_r} ) \geq k_r-1.
\end{equation}
To see this inequality, one notes that since $k_r$ is the largest index for which
$v_{k_r} \leq j_r$, there must be $k_r -1$ terms of the sequence $c_n$ that are 3 or 4 and
come before $u_{j_r}$.

Combining (\ref{second ub}), (\ref{estimate 2}), (\ref{estimate 4}), and (\ref{estimate 5}) gives
\begin{equation}\label{estimate 6}
e(G_{n_r} ) \leq \frac{1}{2} j_r (u_{j_r } - k_r )
- \frac{1}{2} ( l_r - 1)^2 + j_r k_r + k_r l_r
- \frac{1}{2} ( (j_r + 1)^2 + (k_r + 1)^2 + (l_r + 1)^2 ) + o(u_{j_r}^2 )
\end{equation}
for any $r \in \mathbb{N}$.  We divide through by $n_r^2 = u_{j_r}^2$ to get
\begin{equation}\label{estimate 7}
\dfrac{ e (G_{n_r} ) }{ n_r^2 } \leq \dfrac{1}{ u_{j_r}^2 } \left(
\frac{ j_r u_{j_r} }{2} + \frac{ k_r j_r }{2} - \frac{1}{2} (l_r - 1)^2 + k_r l_r - \frac{1}{2} (j_r^2 + l_r^2 + k_r^2 ) \right) + o(1).
\end{equation}

Since $l_r \leq k_r$, we can write $l_r = \epsilon_r k_r$ for some $0 \leq \epsilon _r \leq 1$.
The remaining analysis does not depend on $r$ and so, for ease of notation, we
\begin{center}
{\it omit all occurrences of $r$ as a subscript on all terms.}
\end{center}
Using $l = \epsilon k$, the inequality (\ref{estimate 7}) can be rewritten as
\begin{equation}\label{estimate 8}
\frac{ e(G_n) }{ n^2} \leq \frac{1}{ u_j^2 }
\left( \frac{ j u_j }{2} + \frac{kj}{2} - \frac{j^2}{2} - k^2 ( \epsilon^2 - \epsilon + \frac{1}{2} ) + l \right) + o(1).
\end{equation}
The minimum value of $f( \epsilon ) = \epsilon^2 - \epsilon + \frac{1}{2}$ with $0 \leq \epsilon \leq 1$ is $\frac{1}{4}$ and occurs
when $\epsilon = \frac{1}{2}$.  Thus (\ref{estimate 8}) together with (\ref{estimate 3}) implies
\begin{equation*}
\frac{ e(G_n)}{n^2} \leq \frac{ 1}{u_j^2} \left( \frac{ j u_j }{2} + \frac{kj}{2} - \frac{j^2}{2} - \frac{k^2}{4}  \right) + o(1).
\end{equation*}
We can rewrite this as
\begin{equation*}
\frac{ e(G_n) }{n^2} \leq \frac{1}{u_j^2} \left( \frac{j u_j}{2} - \frac{j^2}{2} + \frac{j}{2} k ( 1  - \frac{k}{2j} ) \right) + o(1).
\end{equation*}
The maximum value of $g(k) = k (1 - \frac{k}{2j} )$ with $0 \leq k \leq j$ is $\frac{j}{2}$ and occurs when $k = j$.  Therefore,
\begin{equation*}
\frac{ e(G_n) }{n^2} \leq  \frac{1}{u_j^2} \left( \frac{j u_j}{2} - \frac{j^2}{2} + \frac{j^2}{4}  \right) + o(1) =
\frac{1}{ u_j^2} \left( \frac{ j u_j}{2} - \frac{j^2}{4} \right) + o(1) .
\end{equation*}
We now have
\begin{equation*}
\frac{ e(G_n) }{n^2} \leq \frac{1}{u_j^2} \left( \frac{ju_j}{2} - \frac{j^2}{4} \right) + o(1) = \frac{1}{4} -
\frac{1}{4} \left( \frac{j}{u_j} - 1 \right)^2 + o(1) \leq \frac{1}{4} + o(1).
\end{equation*}
This shows that for any $r \in \mathbb{N}$, $e(G_{n_r} ) \leq n_r^2 \left( \frac{1}{4} + o(1) \right)$ where
$o(1) \rightarrow 0$ as $r \rightarrow \infty$.  We conclude that
\[
\liminf_{n \rightarrow \infty} \frac{ e(G_n) }{n^2} \leq \frac{1}{4}.
\]
This completes the proof of Theorem \ref{4 upper bound}.



\section{Appendix}

In this section we give the Mathematica \cite{mathematica} code that is used to evaluate the optimization problem from Lemma \ref{count 3} once we have chosen a matrix $D$.  The function \textbf{lowerbound} depends on four inputs $d$, $m$, $l$, and $k$.  The $d$ represents the matrix $D$ in Lemma \ref{count 3} and the $m$ represents the $M$ in Lemma \ref{count 3}.  The input $l$ is the number
of columns of the input $d$ and the input $k$ is the number of rows of $d$.  While $d$ determines $m$, we found it
easier to use a command to produce $m$ first, and then enter this as an input into \textbf{lowerbound}.   The code
for obtaining $m$ from $d$ is
\begin{verbatim}
m = Table[(Sum[d[[a,i]](Sum[d[[b, j]],{b,a+1,k}]),{a,1,k-1}]),
{i,l},{j,l}]
\end{verbatim}

The code for \textbf{lowerbound} is

\begin{verbatim}
lowerbound[d_, m_, l_, k_] := Min[
  Table[
   Minimize[
     {
      (
        (1/3) ( Sum[ m[[i, j]], {i, l}, {j, l} ] ) +
          (1/3) ( Sum[ m[[i, j]] , {i, l - 1}, {j, i + 1, l}] ) +
          (1/6) ( Sum[ m[[i, i]] , {i, l} ] ) +
           Sum[

          Sum[ m[[i, y]], {i, 1, l}] + (1/2) m[[y, y]] +
           Sum[ m[[i, y]], {i, 1, y - 1} ] , {y, 1, t - 1}
                ] +
         x ( Sum[ m[[i, t]] , {i, l} ] ) + x^2 (1/2) m[[t, t]] +
         x ( Sum[ m[[i, t]], {i, t - 1} ] )
        )/
       (
        ( Sum[ d[[i, j]] , {i, k}, {j, l} ] +
           Sum[Sum[ d[[i, y]], {i, 1, k}], {y, 1, t - 1}] +
           x ( Sum[ d[[i, t]] , {i, k} ] )
          )^2
         )
      , 0 <= x <= 1 }, {x} ][[1]] , {t, 1, l}]]
  \end{verbatim}

\end{document}